\documentclass[a4paper]{amsart}
\usepackage[english]{babel}
\usepackage{booktabs}
\usepackage{longtable}
\usepackage[T1]{fontenc}

\newcommand{\Z}{\mathbb{Z}} %Integers
\newcommand{\Q}{\mathbb{Q}} %Rationals
\newtheorem{theo}{Theorem}

\theoremstyle{definition}
\newtheorem*{defn}{Definition}

\theoremstyle{remark}
\newtheorem*{rem}{Remark}

\begin{document}

\title[Elliptic curves induced by Diophantine triples]
{Elliptic curves induced by Diophantine triples}

\author[A. Dujella]{Andrej Dujella}
\address{
Department of Mathematics\\
University of Zagreb\\
Bijeni{\v c}ka cesta 30, 10000 Zagreb, Croatia
}
\email[A. Dujella]{duje@math.hr}

\author[J. C. Peral]{Juan Carlos Peral}
\address{
Departamento de Matem\'aticas\\
Universidad del Pa\'{\i}s Vasco\\
Aptdo. 644, 48080 Bilbao, Spain
}
\email[J. C. Peral]{juancarlos.peral@ehu.es}

\thanks{The first author was supported by the Croatian Science Foundation
under the project no. 6422 and by the QuantiXLie Centre of Excellence, a project
 co-financed by the Croatian Government and European Union through the
 European Regional Development Fund - the Competitiveness and Cohesion
 Operational Programme (Grant KK.01.1.1.01.0004).
 The second author     was supported by  the spanish grant   MINECO R14/P78.       }


\begin{abstract}
Given a  Diophantine triple  $\{ c_1(t),c_2(t), c_3(t)\}$, the
elliptic curve over $\Q (t)$  induced by this triple, i.e.
\[
y^2= ( c_1(t)\, x+1) ( c_2(t)\, x+1) ( c_3(t)\, x+1),
\]
can have as torsion group  one of the non cyclic groups in Mazur's theorem, i.e.
$\Z /  2 \Z\times\Z / 2 \Z$,  $\Z /  2 \Z\times\Z / 4 \Z$,  $\Z /  2 \Z\times\Z / 6 \Z$ or  $\Z /  2 \Z\times\Z / 8 \Z$.
  In this paper we present results concerning the rank over $\Q (t)$ of these curves improving in some of the cases the previously known results.

\end{abstract}

\maketitle

\section{Diophantine quadruples and elliptic curves}
%%%%%%%%%%%%%%%%%%%%%%
\begin{defn}
 A set $\{ c_1, c_2,\ldots, c_m    \}$ of non-zero integers (rationals) is  called
 a (rational) $D(n)$-$m$-tuple if $c_i\cdot c_j+n$ is a perfect square for all $1\le i < j\le m$.
 A $D(1)$-$m$-tuple is also called a Diophantine $m$-tuple.
\end{defn}

Let $\{  c_1, c_2, c_3, c_4 \}$  be a rational Diophantine quadruple. Consider a subtriple  $\{c_1, c_2 , c_3\}$ and define an  elliptic curve by the equation
\begin{equation}\label{dioph}
\tag{$E$} y^2=( c_1 \,x+ 1)( c_2\, x+1) (c_3\, x +1).
\end{equation}
We say that $E$ is the  elliptic curve induced by the Diophantine triple $\{  c_1, c_2, c_3 \}$.
Let
\[
c_i c_j+1=t_{i,j}^2,\quad 1\le i<j\le4.
\]
Then the curve $E$ has the following three rational   points of order $2$:
\[
T_1=[\,-1/c_1,0\,],\qquad T_2=[\,-1/c_2,0\,],\qquad
T_3=[\,-1/c_3,0\,],
\]
and at least three other rational points:
\begin{equation}\label{ratpoints}
\left\{\begin{aligned}
P_1&=[\,0,1\,],\\
P_2&=[\,c_4, t_{1,4}\,t_{2,4}\,t_{3,4}\,],\\
P_3&=\Bigl[\,\frac{ t_{1,2}\,t_{1,3}+ t_{1,2}\,t_{2,3}+ t_{1,3}\,t_{2,3} +1}{c_1 c_2 c_3},\\
&\qquad\frac{(t_{1,2}+t_{1,3})(t_{1,2}+t_{2,3})(t_{1,3}+t_{2,3})}{ c_1 c_2 c_3}\,\Bigr].
\end{aligned}\right.
\end{equation}

Our goal in this article is to   show that Diophantine triples and quadruples are good tools in the search for high rank elliptic curves having as torsion group one of the non-cyclic groups in Mazur's theorem.

In the case of torsion $\Z /  2 \Z\times\Z / 2 \Z$  we show that adequate specialization of the parameters induce subfamilies of curves with rank~$4$, rank~$5$ and rank~$6$ and torsion group $\Z /  2 \Z\times\Z / 2 \Z$.   We use a set of quadruples presented in \cite{D1}. The families of rank~$6$ over $\Q(t)$ represent an
improvement over the known results of curves induced by Diophantine triples.
Namely, in \cite{ADP} a family of rank $\geq 5$ over $\Q(t)$ was constructed.
In the general case (not conditioned to be induced by Diophantine triples)  Elkies constructed a family with rank $\geq 7$, see \cite{D2}.

In our paper  \cite{DP1}  we constructed a curve over $\Q(t)$  induced by Diophantine triples   having rank $4$  and torsion  group  is $\Z /  2 \Z\times\Z / 4 \Z$. We also show an example of a  curve with rank $9$. These are the highest-rank results  known for that torsion.  We  will give an overview of these results in order to complete the presentation.

In the case  of torsion $\Z /  2 \Z\times\Z / 6 \Z$  we present three  curves over $\Q(t)$  with rank $2$ induced by Diophantine triples.  This ties the results for the  general case,  see \cite{DP2}. We also show that the curve with rank $6$ over $\Q$ (which is the record for this torsion group and was found by  Elkies  see \cite{D2}) is  induced by a Diophantine triple.

Finally in the case of torsion $\Z /  2 \Z\times\Z / 8 \Z$ it is known that all such curves are induced by Diophantine triples.

Let us mention that in \cite{DJBS} the authors constructed elliptic curves induced
by Diophantine triples with torsion groups
$\Z/2\Z \times \Z/10\Z$, $\Z/2\Z \times \Z/12\Z$ and $\Z/4\Z \times \Z/4\Z$
over quadratic fields.

\begin{rem} It has been conjectured in the past by many authors that there exist elliptic curves of arbitrarily high rank, and even for each of the torsion groups in Mazur's theorem.

However there are also heuristic arguments that suggest the boundedness of the rank of elliptic curves, see section $8.3$ in \cite{PPVW}, for example. According to this heuristic only a finite number of curves would have rank higher that $21$. The known lower bounds for infinite families fit very well with the bounds given  by these heuristic arguments, see \cite{D2}.

In particular for the following torsion groups the conjectured bound by the heuristic and the known bounds coincide:
$\Z / 6 \Z$  (bound $5$), $\Z / 8 \Z$  (bound $3$), $\Z/2\Z \times \Z/4\Z$  (bound $5$) and $\Z/2\Z \times \Z/6\Z$  (bound $3$).
\end{rem}

\section{Torsion   $\Z /  2 \Z\times\Z / 2 \Z$ and rank $6$ over $\Q(t)$}\label{sec2}
%%%%%%%%%%%%%%%%%%%%%%%%%%%%%%%%%%%%%%%%%%%%%%%%%%%%%%%%%%

\subsection{Rank $\geq 3$ over $\Q(t)$. } 
Since an elliptic curve induced by a Diophantine triple has three rational points of order $2$, 
its torsion group contains $\Z/2\Z \times \Z/2\Z$. 

In \cite{D1}   several families of $D(n)$-quadruples are described.
We  will use for our construction the one given by
\[
\{a, a(k+1)^2-2\,k, a(2\,k+1)^2-8\,k-4, a\,k^2-2\,k-2\}.
\]
For each $a$ and $k$ this quadruple is a $D(2\,a(2\,k+1)+1)$-quadruple.
Now we  specialize to the following value of  $k$:
\[
k =\frac{-1-2\,a+n^2}{4\,a}.
\]
The resulting quadruple is a $D(n^2)$-quadruple and once divided by $n$ we get the following rational $D(1)$-quadruple:
\begin{equation}\label{c_quad}
\left\{\begin{aligned}
c_1(a,n)&=\dfrac{a}{n},\\
c_2(a,n)&=\dfrac{((n-3)(n-1)+2\,a)((n+1)(n+3)+2\,a)}{16\,a\,n},\\
c_3(a,n)&=\dfrac{(n-3)(n-1)(n+1)(n+3)}{4\,a\,n},\\
c_4(a,n)&=\dfrac{((n-3)(n-1)-2\,a)((n+1)(n+3)-2\,a)}{16\,a\,n}.
\end{aligned}\right.
\end{equation}
In the terminology of \cite{G}, \eqref{c_quad} is an irregular and twice semi-regular Diophantine quadruple.
A Diophantine triple $\{a_1,a_2,a_3\}$ is regular if $(a_3-a_2-a_1)^2 = 4(a_1a_2+1)$, while a Diophantine quadruple
$\{a_1,a_2,a_3,a_4\}$ is regular if $(a_4+a_3-a_1-a_2)^2 = 4(a_1a_2+1)(a_3a_4+1)$. It can be checked that
\eqref{c_quad} is irregular, but it contains two regular triples:
$\{c_1,c_2,c_4\}$ and $\{c_2,c_3,c_4\}$.

Now we define the  elliptic curve associated to the triple $\{ c_1, c_2, c_3\}$ as explained above, i.e.:
\[
y^2=( c_1(a,n) x+1)(c_2(a,n) x+1)( c_3(a,n) x+1).
\]
Note that we choose an irregular triple which is a subtriple of an irregular quadruple.
Otherwise, by \cite{D3}, the points $P_1$, $P_2$, $P_3$ would not be independent.

Besides the $2$-torsion points, this curve has the points with $x$-coordinate given by
\[
0,\quad c_4(a,n)\quad\text{and}\quad\frac{t_{1,2}\,t_{1,3}+ t_{1,2}\,t_{2,3}+ t_{1,3}\,t_{2,3} +1}{c_1(a,n)c_2(a,n)c_3(a,n)},
\]
where as before $t_{i,j}=t_{i,j}(a,n)=\sqrt{c_i(a,n) c_j(a,n)+1}$, $1\le i<j\le3$. In terms of $a$ and $n$, the  three rational points~\eqref{ratpoints} are:
\begin{align*}
P_1&=[\,0,1\,],\\
P_2&=\Bigl[\frac{(n^2+4\,n-2\,a+3)(n^2-4\,n-2\,a+3)}{16\,a\,n},\\
&\qquad-\frac{(n^2-2\,a+3)(n^4-10\,n^2-4\,a^2+9)(n^4-2\,a\,n^2-10\,n^2-6\,a+9)}{512\,a^2\,n^3}\Bigr],\\
P_3&=\Bigl[\frac{6\,n}{(n-3)(n+3)},\frac{(n^2+6\,a-9)(3\,n^2+2\,a-3)}{4\,a(n-3)(n+3)}\Bigr].
\end{align*}

\begin{theo}\label{teo:1}
The curve $y^2=( c_1(a,n) x+1)( c_2(a,n) x+1)( c_3(a,n) x+1)$
has torsion group containing $\Z/2\Z\times\Z/2\Z$ and rank~$\geq 3$ over $\Q(n,a)$.
The  points $P_1$, $P_2$ and $P_3$ are of infinite order and independent.
\end{theo}
\begin{proof}

Since the specialization map  is always a homomorphism, see  e.g.~\cite{S}, it is enough to prove that
there exist values of $a$ and $n$ such that the specialized three points are $\Q$-independent.
Consider for example  $a=2$ and $n=5$. Then  the specialized points are
\[
Q_1=[0,1] ,\quad Q_2=[11/10, -1173/125],\quad Q_3=[15/8, 133/8].
\]
A calculation using J.~Cremona's program \texttt{mwrank} \cite{mwrank} shows that the elliptic curve
induced by the triple having these parameters has rank $3$, and from obtained generators it is easy to check that the three points $Q_1$, $Q_2$ and $Q_3$ are independent.
%Thus, by the specialization theorem of Silverman, the proof is finished.
\end{proof}

The symbolic calculations in this and the next sections were carried out with \textit{Mathematica}\textsuperscript\circledR \cite{Math}.

\subsection{Search for higher rank}\label{subsec2.2}

Now we look for conditions on $a$ and $n$ such that there are new rational points on the curve. This task is made simpler by means of a change of variable. The coordinate transformation
\[
x\mapsto c_1(a,n) c_2(a,n) c_3(a,n)\,x,\quad y\mapsto c_1(a,n) c_2(a,n) c_3(a,n)\,y
\]
applied to the curve leads to
\[
y^2=(x+c_1(a,n) c_2(a,n))(x+c_1(a,n) c_3(a,n))(x+c_2(a,n) c_3(a,n)).
\]
Next, the change $x\mapsto x-c_1(a,n) c_2(a,n)$ transforms it into
\begin{multline*}
y^2=x(x+c_1(a,n) c_3(a,n)-c_1(a,n) c_2(a,n))\\
\times(x+c_2(a,n) c_3(a,n)- c_1(a,n) c_2(a,n)).
\end{multline*}
From this point on, in order to avoid denominators, we will make, when necessary, the appropriate change of variables to write the curve as
\begin{equation}\label{ec}
y^2= x^3+A\,x^2+B\,x
\end{equation}
where $A(a,n)$ and $B(a,n)$ are integral. This leads to the following values of the coefficients $A$ and $B$:
\begin{align*}
A(a,n)&=81+108\,a+108\,a^2-96\,a^3-32\,a^4-180\,n^2-84\,a\,n^2-120\,a^2\,n^2\\
&\qquad-32\,a^3\,n^2+118\,n^4-28\,a\,n^4+12\,a^2\,n^4-20\,n^6+4\,a\,n^6+n^8,\\
B(a,n)&=4\,a^2(9+2\,a-n^2)(3+2\,a-4\,n+n^2)(3+2\,a+4\,n+ n^2)\\
&\qquad\times(-3+2\,a+3\,n^2)(-9+4\,a^2+10\,n^2-n^4).
\end{align*}

%Finally, the $x$-coordinates of the three infinite order points are
Finally, the $x$-coordinates of the three rational points corresponding (under the coordinate
transformation) to $P_1$, $P_2$ and $P_3$ are
\begin{align*}
x_1&=4\,a^2(3+2\,a-4\,n+n^2)(3+2\,a+4\,n+n^2),\\
x_2&=\frac{(3+2\,a-4\,n+n^2)(3+2\,a+4\,n+n^2)(9-6\,a-10\,n^2-2\,a\,n^2+n^4)^2}{16\,n^2},\\
x_3&=2\,a(3+2\,a-4\,n+n^2)(3+2\,a+4\,n+n^2)(-3+2\,a+3\,n^2).
\end{align*}

\subsection{Construction of a curve of rank $4$ over $\Q(t)$}\label{subr4}

Now we look for those  polynomial factors of $B$ that can be conditioned in a simple way to yield a new point in the curve.

The condition for
$(3 + 2 a - 4 n + n^2) (-3 + 2 a + 3 n^2) (-9 + 4 a^2 + 10 n^2 - n^4)$
to became the $X$ coordinate of a new point   is that
 $2 (9 + 6 a + 8 a^2 - 18 n -
   4 a n + 8 n^2 - 2 a n^2 + 2 n^3 - n^4)$
   converts into a square.  This can be achieved with the value $n=7/3$.
   The coefficients of the curve are
   \begin{align*}
A(a)&=-2 (-51200 + 109440 a + 38880 a^2 + 55404 a^3 + 6561 a^4)\\
B(a)&=243 a^2 (20 + 3 a) (-4 + 9 a) (16 + 9 a) (80 + 9 a) (320 + 81 a^2),
\end{align*}
and the  $x$-coordinates of the preceding points jointly with the new one are
\begin{equation}\label{xcoord}
\left\{
\begin{aligned}
x_1&= 81 a^2 (-4 + 9 a) (80 + 9 a)\\
x_2&=27 a (20 + 3 a) (-4 + 9 a) (80 + 9 a)\\
x_3&=\frac{1}{441} (-4 + 9 a) (80 + 9 a) (160 + 171 a)^2\\
x_4&=3 (20 + 3 a) (-4 + 9 a) (320 + 81 a^2).
\end{aligned}
\right.
\end{equation}
This is a curve with rank $\geq 4$ over  $\Q(a)$ since it can be proved that the four points quoted above  are independent. Indeed, by taking $a=7$ we find that the specialized points are independent.
Moreover, this specialization satisfies the conditions of \cite[Theorem 1.1]{GT2} and thus shows that
the rank over $\Q(a)$ is exactly equal to $4$ and the torsion group is equal to $\Z/2\Z \times \Z/2\Z$.

The corresponding quadruple is
\begin{equation}\label{quadruple}
\left\{
\begin{aligned}
q_1&=-\dfrac{3 a}{7},\\
q_2&=-\dfrac{(80 + 9 a) (-4 + 9 a)}{756\, a},\\
q_3&=\dfrac{320}{189\, a},\\
q_4&=-\dfrac{(4 + 9 a) (-80 + 9 a)}{756\, a}.
\end{aligned}
\right.
\end{equation}

\begin{theo}\label{teo:2}
The elliptic curve induced by the first three components of the Diophantine quadruple~\eqref{quadruple}
has torsion $\Z/2\Z\times  \Z/2\Z$ and  rank equal to $4$  over $\Q(a)$.
The  points with $x$-coordinate given in~\eqref{xcoord} are of infinite order and independent.
\end{theo}

\subsection{Construction of curves of   rank $\geq 5$  over $\Q(t)$}\label{subr5}

We have found $30$ specializations of the parameter in the above rank $4 $ curve leading to  curves having rank $\geq 5$.  In four cases with  a further specialization we get  rank $6$  over $\Q(t)$.

We present the details in one of the cases and in the other three  we just  quote the specialization of the parameter $a$ leading to  rank $6$ curves.

In order to force $9 a (-20 + 9 a) (16 + 9 a) (80 + 9 a)$ as $X$-coordinate of a new point in the rank $4$ curve we have to parametrize $10 (-20 + 9 a) (-2 + 9 a)=$ square. We get
$$
a=\frac{2 (-1 + 10 w) (1 + 10 w)}{9 (-1 + 10 w^2)}
$$
Now the curve with rank $\geq 5$  is $Y^2=X^3+ A(w) X^2+ B(w) X$ where
 \begin{align*}
A(w)&=-2 (-169 - 12020 w^2 + 678000 w^4 - 12680000 w^6 + 80000000 w^8),\\
B(w)&=(-1 + 10 w)^2 (1 + 10 w)^2 (-1 + 20 w^2) (1 + 80 w^2) (-31 +
   400 w^2) \\
   &(-41 + 500 w^2) (9 - 200 w^2 + 2000 w^4).
\end{align*}
Five independent points have the $X$-coordinates as follows
\begin{equation}\label{xcoord5}
\left\{
\begin{aligned}
x_1&= \frac{1}{9}(-1 + 10 w)^2 (1 + 10 w)^2 (1 + 80 w^2) (-41 + 500 w^2),\\
x_2&=\frac{1}{9}(-1 + 10 w) (1 + 10 w) (1 + 80 w^2) (-31 + 400 w^2) (-41 + 500 w^2),\\
x_3&=\frac{1}{49} (1 + 80 w^2) (-11 + 300 w^2)^2 (-41 + 500 w^2),\\
x_4&=(1 + 80 w^2) (-31 + 400 w^2) (9 - 200 w^2 + 2000 w^4),\\
x_5&=9 (-1 + 10 w) (1 + 10 w) (-1 + 20 w^2) (-41 + 500 w^2),\\
\end{aligned}
\right.
\end{equation}
and the quadruple is:
\begin{equation}\label{quadruple5}
\left\{
\begin{aligned}
q_1&=-\dfrac{2 (-1 + 10 w) (1 + 10 w)}{21 (-1 + 10 w^2)},\\
q_2&=-\dfrac{(1 + 80 w^2) (-41 + 500 w^2)}{42 (-1 + 10 w) (1 + 10 w) (-1 + 10 w^2)},\\
q_3&=\dfrac{160 (-1 + 10 w^2)}{21 (-1 + 10 w) (1 + 10 w)},\\
q_4&=\dfrac{3 (-1 + 40 w^2) (-13 + 100 w^2)}{14 (-1 + 10 w) (1 + 10 w) (-1 + 10 w^2)}.\\
\end{aligned}
\right.
\end{equation}

\subsection{Construction of curves of   rank $6$  over $\Q(t)$}

Now we can force  $-9 (-1 + 10 w)^2 (1 + 10 w)^2 (-1 + 20 w^2) (-41 + 500 w^2)$ as $X$ coordinate of a new point in the previous curve with rank $\geq 5$ by solving $-(-2 + 7 w) (2 + 7 w)=$ square, so we have
$$w=\frac{2 (-1 + v) (1 + v)}{7 (1 + v^2)}$$

With this choice of $w$ the curve transforms into the rank $6$ curve given by $Y^2=X^3+ A(v) X^2+ B(v) X$ where
 \begin{align*}
A(v)&=-2 (130752711 - 35202346632 v^2 + 260292593988 v^4 -
   1337869740984 v^6\\& + 1975889131370 v^8 - 1337869740984 v^{10} +
   260292593988 v^{12} - 35202346632 v^{14}\\& + 130752711 v^{16}),\\
B(v)&=-(9 - 80 v + 9 v^2) (9 + 80 v + 9 v^2) (-27 + 13 v^2)^2 (-13 +
   27 v^2)^2 \\&
   (9 + 8018 v^2 + 9 v^4) (31 - 258 v^2 + 31 v^4) (369 -
   542 v^2 + 369 v^4)\\&
    (14409 - 41564 v^2 + 400054 v^4 - 41564 v^6 +
   14409 v^8)).
\end{align*}
The $X$-coordinates of six independent points are as follows
\begin{equation}\label{xcoord6}
\left\{
\begin{aligned}
x_1&= -\frac{1}{9}(-27 + 13 v^2)^2 (-13 + 27 v^2)^2 (9 + 8018 v^2 + 9 v^4)\\& (369 -
   542 v^2 + 369 v^4),\\
x_2&=-\frac{1}{9}(9 - 80 v + 9 v^2) (9 + 80 v + 9 v^2) (-27 + 13 v^2) \\
&(-13 + 27 v^2) (9 + 8018 v^2 + 9 v^4) (369 - 542 v^2 + 369 v^4),\\
x_3&=-\frac{1}{49}(9 + 8018 v^2 + 9 v^4) (369 - 542 v^2 + 369 v^4) \\&(661 - 3478 v^2 +
   661 v^4)^2,\\
x_4&=(9 - 80 v + 9 v^2) (9 + 80 v + 9 v^2) (369 - 542 v^2 + 369 v^4) \\
&(14409 - 41564 v^2 + 400054 v^4 - 41564 v^6 + 14409 v^8),   \\
x_5&=-441 (1 + v^2)^2 (-27 + 13 v^2) (-13 + 27 v^2) (9 + 8018 v^2 +
   9 v^4)\\ & (31 - 258 v^2 + 31 v^4),\\
   x_6&=9 (-27 + 13 v^2)^2 (-13 + 27 v^2)^2 (9 + 8018 v^2 + 9 v^4) \\&(31 -
   258 v^2 + 31 v^4),\\
\end{aligned}
\right.
\end{equation}
and the quadruple becomes:
\begin{equation}\label{quadruple6}
\left\{
\begin{aligned}
q_1&=\dfrac{2 (-27 + 13 v^2) (-13 + 27 v^2)}{21 (9 + 178 v^2 + 9 v^4)},\\
q_2&=-\dfrac{(9 + 8018 v^2 + 9 v^4) (369 - 542 v^2 + 369 v^4)}{42 (-27 + 13 v^2) (-13 + 27 v^2) (9 + 178 v^2 + 9 v^4)},\\
q_3&=-\dfrac{160 (9 + 178 v^2 + 9 v^4)}{21 (-27 + 13 v^2) (-13 + 27 v^2)},\\
q_4&=\dfrac{3 (111 - 418 v^2 + 111 v^4) (237 + 2074 v^2 + 237 v^4)}{14 (-27 + 13 v^2) (-13 + 27 v^2) (9 + 178 v^2 + 9 v^4)}.
\end{aligned}
\right.
\end{equation}
By taking the specialization $v=5$ and applying \cite[Theorem 1.1]{GT2}
we see that these six points are independent and the rank over $\Q(v)$ is exactly equal to $6$ 
and the torsion group is equal to $\Z/2\Z \times \Z/2\Z$.

\begin{theo}\label{teo:3}
The elliptic curve induced by the first three components of the Diophantine quadruple~\eqref{quadruple6}
has torsion $\Z/2\Z\times  \Z/2\Z$ and  rank equal to $6$  over $\Q(v)$.
The  points with $x$-coordinate given in~\eqref{xcoord6} are of infinite order and independent.
\end{theo}

\subsection{Other curves of   rank $6$  over $\Q(t)$} The following specializations of the parameter $a$ in the rank $4$  curve of subsection~\eqref{subr4} also produce rank $6$ curves
\begin{equation}\label{cond}
\left\{
\begin{aligned}
a_1&=-\dfrac{64 (831744 - 40128 v + 4288 v^2 - 44 v^3 + v^4)}{9 (-1520 + 88 v + v^2) (-2736 - 264 v + 5 v^2)},\\
a_2&=-\dfrac{10732176 - 628992 v + 19192 v^2 - 576 v^3 + 9 v^4}{36 (-27 + v) v (-364 + 9 v)},\\
a_3&=-\dfrac{5 (-10 + 6 v + v^2) (-18 - 18 v + 5 v^2)}{9 (12 - 2 v + v^2) (3 - v + v^2)}.\\
\end{aligned}
\right.
\end{equation}

\subsection{Examples of curves of  rank $11$ over $\Q$}\label{subsec.7}
Here we give some information on curves induced by Diophantine triples with rank equal to $11$.
At present there are only few curves known with torsion group
$\Z/ 2\Z\times\Z/2\Z$ and rank~$\geq 11$ (due to Elkies, Eroshkin, Dujella and Kulesz, and the authors;
the record is a curve with rank~$15$ found by Elkies in 2009; see \cite{D2} for details).
We were able to find six new curves with rank 11 in the family from Section \ref{subr4}
for $a=-\frac{2020}{112311}$, $\frac{12130}{1349}$, $-\frac{12792}{4825}$,
$\frac{40180}{111879}$, $\frac{144050}{43983}$, $\frac{16648}{12463}$,
and one curve in the family from Section \ref{subr5} for $w=409/400$.
The curves are found by the similar sieving algorithm as in \cite{ADP},
and the rank is computed by \texttt{mwrank} \cite{mwrank}.
We give details only for the first curve corresponding to $a=-\frac{2020}{112311}$,
since it is the curve with the smallest conductor among all known curves
with torsion $\Z/ 2\Z\times\Z/2\Z$ and rank 11. Details on other mentioned curves
can be found on \cite{D2}.

The curve
\begin{multline*}
y^2=x^3+2578115108336702464000\,x^2 \\
-1552915744756258927995988436385792000000\,x,
\end{multline*}
has rank $11$ over $\Q$ and it is
induced by the rational Diophantine triple
\[
\Bigl\{\,\frac{2020}{262059},\, -\frac{215599320}{8822653},\, -\frac{199664}{2121}\,\Bigr\}.
\]
The minimal Weierstrass equation for this curve is
\begin{multline*}
y^2 = x^3 + x^2 - 3383044565362668792508854542324233 \,x \\
         \mbox{}+ 75737263950801942103904570668638566598532431572663
\end{multline*}
Torsion points are $\mathcal{O}$ and:
\begin{gather*}
[33569207139800813, 0], [-67161937901586997, 0], [33592730761786183, 0].
\end{gather*}
Independent points of infinite order are:
\begin{align*}
Q_1 &=[21966937543749263, 3467304529380020399592600], \\
Q_2 &=[33560959657313363, 5137371389991712861800], \\
Q_3 &=[34582038915915083, 319292566508771961823200], \\
Q_4 &=\Bigl[\frac{134483428957332077}{4}, \frac{96800241541200000930525}{8}\Bigr], \\
Q_5 &=[31780818196317203, 566228347995386167907400], \\
Q_6 &=\Bigl[\frac{1413627031281777887}{49}, \frac{502860429284713027048333800}{343}\Bigr], \\
Q_7 &=\Bigl[\frac{9700355475181684787}{289}, \frac{16274604587975225427493800}{4913}\Bigr], \\
Q_8 &=[33623189883813803, 12873137085655241194800], \\
Q_9 &=[33548281734503453, 9678428398238096687400], \\
Q_{10} &=\Bigl[\frac{173923494287946441475667}{2809}, \frac{72533258254815185975583870670539600}{148877}\Bigr], \\
Q_{11} &=\Bigl[-\frac{2251179780875766539120038}{34421689}, \frac{839055527133163935536879014270829625}{201952049363}\Bigr],
\end{align*}
so that its rank is at least $11$. \texttt{mwrank} (which uses 2-descent, via 2-isogeny if possible,
to unconditionally determine the rank) establishes that in fact the rank is exactly equal to $11$.

\section{Torsion   $\Z /  2 \Z\times\Z / 4 \Z$ and rank $4$ over $\Q(t)$}\label{sec3}

\subsection{Curve with torsion $\Z /  2 \Z\times\Z / 4 \Z$  and rank $4$ over $\Q(t)$ }
We just summarize  the main results in our paper \cite{DP1}. In that paper we have shown that the following triple
\begin{align*}
 a=& -\frac{(t+1) (31t^4+52t^3+22t^2-4t-1) (3t^2+2t+1)}{t(11t^4+12t^3+2t^2-4t-1) (9t^2+14t+7)}, \\
 b= & \frac{t(11t^4+12t^3+2t^2-4t-1) (9t^2+14t+7)}{(t+1) (31t^4+52t^3+22t^2-4t-1) (3t^2+2t+1)}, \\
 c =& 	\big(16(t-1)(3t+1)(t+1)t(t^2+6t+3)(3t^2+6t+1)\\ &(5t^2+2t-1)(7t^2+2t+1)\big)/\\&
((11t^4+12t^3+2t^2-4t-1)(9t^2+14t+7)\\&(31t^4+52t^3+22t^2-4t-1)(3t^2+2t+1))
 \end{align*}
  induces the  elliptic curve $$ E:\quad y^2 = x^3 + A(t)x^2 + B(t)x, $$ where\begin{align*}
 A(t)&=\\
  &2(87671889 t^{24}+854321688 t^{23}+3766024692 t^{22}+9923033928 t^{21}\\
  &+17428851514  t^{20}+21621621928 t^{19} +19950275060 t^{18}\\
  &+15200715960 t^{17}+11789354375 t^{16}
 +10470452464 t^{15}+8925222696 t^{14}\\
 &+5984900048 t^{13}+2829340620 t^{12}
 +820299856 t^{11}+59930952 t^{10}\\
 &-66320528 t^9
-35768977 t^8-9381000 t^7
 -1017244 t^6+262760  t^5\\
 &+159130 t^4+41096  t^3+6468t^2+600t+25), \\
 B(t)&=\\ &(t^2-2t-1)^2 (69t^4+148t^3+78t^2+4t+1)^2 (13t^2-2t-1)^2 \\
 & \times (9t^4+28t^3+18t^2+4t+1)^2 (11t^4+12t^3+2t^2-4t-1)^2 \\
 & \times (9t^2+14t+7)^2 (31t^4+52t^3+22t^2-4t-1)^2 (3t^2+2t+1)^2,
 \end{align*}
 having rank $\geq 4$ over $\mathbb{Q}(t)$. Indeed, it contains the points whose $X$-coordinates are
 \begin{eqnarray*}
 x_1 &\!\!=\!\!& (9t^4+28t^3+18t^2+4t+1)^2 (11t^4+12t^3+2t^2-4t-1)^2 \\
 & & \mbox{}\times (69t^4+148t^3+78t^2+4t+1)^2,\\
 x_2 &\!\!=\!\!& (3t^2+2t+1) (9t^2+14t+7)^2 (13t^2-2t-1)\\
 & & \mbox{}\times (9t^4+28t^3+18t^2+4t+1)(11t^4+12t^3+2t^2-4t-1)^2 \\
 & & \mbox{}\times (31t^4+52t^3+22t^2-4t-1), \\
 x_3 &\!\!=\!\!& (3t^2+2t+1) (9t^2+14t+7)^2 (13t^2-2t-1)  \\
& & \mbox{}\times (9t^4+28t^3+18t^2+4t+1)^2 (11t^4+12t^3+2t^2-4t-1)  \\
 & & \mbox{}\times  (69t^4+148t^3+78t^2+4t+1),  \\
 x_4 &\!\!=\!\!& -(3t^2+2t+1)^2 (9t^2+14t+7)^2 (11t^4+12t^3+2t^2-4t-1)^2 \\
 & & \mbox{}\times (31t^4+52t^3+22t^2-4t-1)^2,
 \end{eqnarray*}
 and a specialization, e.g. $t=2$, shows that the four points $P_1, P_2, P_3, P_4$, having these $X$-coordinates, are independent points of infinite order. Thus we obtained an elliptic curve over the field of rational functions with torsion group containing $\Z/2\Z \times \Z/4\Z$ and rank $\geq 4$.
 In fact, by using the algorithm of Gusi\'c and Tadi\'c \cite{GT1,GT2} to find an appropriate  
 injective specialization (e.g. $t=15)$, it can be shown that the rank over $\Q(t)$ is exactly equal to $4$ and the torsion group is equal to $\Z/2\Z \times \Z/4\Z$.
 This improves previous records (with rank $\geq 3$) for curves with this torsion group, obtained by Lecacheux \cite{L}, Elkies \cite{El} and Eroshkin \cite{Er}.

\subsection{  A curve with torsion $\Z /  2 \Z\times\Z / 4 \Z$ and rank $9$ over $\Q$}

In the same paper \cite{DP1} we presented a curve of rank $9$ over $\Q$. This is the current record for all known elliptic curves with torsion group  $\mathbb{Z}/2\mathbb{Z} \times \mathbb{Z}/4\mathbb{Z}$. Previous records with rank $8$, are due to Elkies, Eroshkin and Dujella (\cite{El,Er}).

 The minimal Weierstrass form of the curve is
 \begin{eqnarray*}
 & y^2 = x^3 + x^2 - 6141005737705911671519806644217969840\,x \\
 &   \mbox{} + 5857433177348803158586285785929631477808095171159063188.
 \end{eqnarray*}
  It is induced by the Diophantine triple $$ \Bigl\{ -\frac{126555}{2686},\, \frac{2686}{126555},\, -\frac{9107022944}{249946125} \Bigr\}. $$
The curve is obtained from the family of curves
\begin{equation*} \label{alphaT}
   y^2 = x^3 + 2 (\alpha^2 + \tau^2 + 4 \alpha^2 \tau^2 + \alpha^4 \tau^2 + \alpha^2 \tau^4)x^2+(\tau+\alpha)^2 (\alpha \tau-1)^2 (\tau-\alpha)^2 (\alpha \tau+1)^2 x,
   \end{equation*}
where
\begin{equation*} \label{XYZV}
  \tau= \frac{r^2-s^2-2t^2+2v^2}{2(r t+s v)}, \quad\alpha = \frac{r s-2 t v}{r t+s v},
\end{equation*}
for $(r,s,t,v)= (155, 54, 96, 106)$.
We obtained in \cite{DP1} also several curves with rank 8, corresponding to the parameters
$(r,s,t,v)=
(20, -11, 25,  68)$, $(82, 9, 73, 30)$, $(55, 31, 142, 15)$, $(91, 55, 33, 104)$, $(157, 127, 43, 12)$.
Here we report on finding five new curves with torsion $\Z /  2 \Z\times\Z / 4 \Z$ and rank $8$,
corresponding to the pa\-ra\-me\-ters
$(r,s,t,v)=(131, -29, 49, 96)$, $(186, -57, 62, 199)$,
$(107, 107, 149, 430)$, $(103, 103$, $168, 725)$, $(749, 749, 138, 245)$.
The details on these curves can be found on \cite{D2}.

 \section{Torsion   $\Z /  2 \Z\times\Z / 6 \Z$ and rank $2$ over $\Q(t)$}\label{sec4}

 \subsection{Torsion   $\Z /  2 \Z\times\Z / 6 \Z$ and rank $\geq 1$}\label{subsec4}

Consider the following quintuple
\begin{align*}
 a=&\dfrac{  (u^3 - 9 u)  }{6 (u^2 - 1 )}, \\
 b=&-\dfrac{ 18 ( u^2 - 1) }{4(u^3 - 9 u)}, \\
 c_1=&\dfrac{(9 - 12 u - 18 u^2 + 4 u^3 + u^4)}{6(u (u^2 + 2 u - 3))},\\
  c_2=&\dfrac{(9 + 12 u - 18 u^2 - 4 u^3 + u^4)}{6(u (u^2 - 2u - 3))},\\
  c_3=&-\dfrac{16u(u^2 - 3)}{3 (u^4 - 10u^2 + 9)  }.
  \end{align*}
  Then $\{a, b, c_1, c_2,
  c_3\}$ is a rational Diophantine quintuple. Furthermore, the Diophantine triples
  $\{a, b, c_1\}$,  $\{a, b, c_2\}$ and $\{a, b, c_3\}$
satisfy the condition of \cite[Lemma 1]{DKMS}, and thus
  the elliptic curves induced by these three triples
have torsion $\Z/2 \Z \times \Z/ 6 \Z$ and rank $\geq 1$ over $\Q (u)$. These three curves are birationally
equivalent.

The elliptic curve induced  by $\{a, b, c_1\}$ can be written as $Y^2=X^3+A(u) X^2+ B(u) X$ where
\begin{align*}
 A(u)=&81 - 162 u + 54 u^2 + 162 u^3 + 162 u^4 - 54 u^5 + 6 u^6 + 6 u^7 + u^8, \\
 B(u)=&-27 (-1 + u)^3 u^3 (3 + u)^3 (6 - 3 u + u^2) (3 + 3 u + 2 u^2).
  \end{align*}
A point of infinite order point is $P(u)$ given by
 \[
[-27 (-1 + u)^2 u^2 (3 + u)^2, 27 (-3 + u) (-1 + u)^2 u^2 (1 + u) (3 + u)^2 (-3 + 6 u + u^2)].
 \]

\subsection{Search for higher rank. First curve}
We have found  that imposing $$\frac{4(-1+u) u^2 (3+u)^2 (6-3 u+u^2) (3+3 u+2 u^2)}{(-3+u)^2}$$ as the
$X$-coordinate of a new point in the curve is equivalent to parametrize $4-u+u^2=$ square, which can be done with $u=-\frac{(-2+w) (2+w)}{1+2 w}$.
Using this specialization of the parameter we get a curve with rank $2$ over $\Q(w)$ and torsion $\Z/2 \Z \times \Z/ 6 \Z$.  The curve can be written as  $Y^2=X^3+ A_1(w) X^2+ B_1(w) X$ where
\begin{align*}
 A_1(w)=&185257 + 401184 w + 530914 w^2 + 1218012 w^3 + 1041238 w^4 -
 925272 w^5\\& - 1369658 w^6 + 53676 w^7 + 519130 w^8 + 93984 w^9 -
 59978 w^{10} - 12924 w^{11} \\&+ 3286 w^{12} + 792 w^{13} - 14 w^{14} -
 12 w^{15} + w^{16}, \\
 B_1(w)=&27 (-7 + w)^3 (-2 + w)^3 (-1 + w)^3 (1 + w)^3 (2 + w)^3 (3 +
   w)^3 (1 + 2 w)^3\\&
    (10 + 19 w^2 + 6 w^3 + w^4) (47 + 36 w - 7 w^2 -
   6 w^3 + 2 w^4).
  \end{align*}
  Two independent points of infinite order have the following $X$-coordinates
  \begin{align*}
 x_1=&-27 (-7 + w)^2 (-2 + w)^2 (-1 + w)^2 (1 + w)^2 (2 + w)^2 (3 +
   w)^2 (1 + 2 w)^2, \\
 x_2=&-\frac{4}{(-1 + 6 w + w^2)^2 } (-7 + w)^2 (-2 + w)^2 (-1 + w) (1 + w)^2 (2 + w)^2\\&
  (3 + w) (1 +
   2 w) (10 + 19 w^2 + 6 w^3 + w^4) (47 + 36 w - 7 w^2 - 6 w^3 +
   2 w^4).
  \end{align*}
 The specialization $w=57$ satisfies the conditions of \cite[Theorem 1.1]{GT2},
 which shows that the rank over $\Q(w)$ is equal to $2$.

  The triple used for the curve is as follows
   \begin{align*}
 q_1=&-\frac{(-7 + w) (-2 + w) (1 + w) (2 + w) (-1 + 6 w + w^2)}{6 (-1 + w) (3 + w) (1 + 2 w) (-5 - 2 w + w^2)} ,\\
 q_2=&-\frac{9 (-1 + w) (3 + w) (1 + 2 w) (-5 - 2 w + w^2)}{ 2 (-7 + w) (-2 + w) (1 + w) (2 + w) (-1 + 6 w + w^2) },\\
    q_3=&-\frac{(-5 - 2 w + w^2) (-1 + 6 w + w^2) (37 + 36 w - 26 w^2 - 12 w^3 + w^4)}{6 (-7 + w) (-2 + w) (-1 + w) (1 + w) (2 + w) (3 + w) (1 + 2 w)}.
  \end{align*}

  We get rank $5$ curves over $\Q$ for the following values: $w=\frac{1}{7}$,  $w=\frac{12}{11}$ and $w=\frac{33}{14}$.

  \subsection{Search for higher rank. Second curve}
  Also we have found that imposing $$\frac{4 (-1 + u) u^3 (3 + u)^3 (3 + 3 u + 2 u^2)}{(1 + u)^2
  }$$
  as the $X$-coordinate of a new point in the curve 
  is the same that solving $u^2+u+2=$ square or $u=-\frac{-2+w^2}{-1+2 w}$. This specialization of the parameter leads to another curve with torsion $\Z/2 \Z \times \Z/ 6 \Z$ and rank $2$ over $\Q(t)$. The coefficients of the curve are
  \begin{align*}
 A_2(w)=&3517 - 26568 w + 74462 w^2 - 102612 w^3 + 138610 w^4 - 283680 w^5\\& +
 346730 w^6 - 107316 w^7 - 122138 w^8 + 84648 w^9 + 7418 w^{10} -
 15084 w^{11} \\&
 + 1690 w^{12} + 576 w^{13} + 14 w^{14} - 12 w^{15} + w^{16}, \\
 B_2(w)=&27 (-1 + w)^3 (3 + w)^3 (-1 + 2 w)^3 (-2 + w^2)^3 (1 - 6 w +
   w^2)^3\\& (16 - 36 w + 17 w^2 + 6 w^3 + w^4) (5 + 7 w^2 - 6 w^3 +
   2 w^4).
  \end{align*}
   Two independent points of infinite order have the following $X$ coordinates
    \begin{align*}
 x_1=&-27 (-1 + w)^2 (3 + w)^2 (-1 + 2 w)^2 (-2 + w^2)^2 (1 - 6 w + w^2)^2, \\
 x_2=&-\frac{4}{(-1 - 2 w + w^2)^2 } (-1 + w) (3 + w) (-1 + 2 w) (-2 + w^2)^3 \\&(1 - 6 w + w^2)^3 (5 +
   7 w^2 - 6 w^3 + 2 w^4).
  \end{align*}
 The specialization $w=18$ satisfies the conditions of \cite[Theorem 1.1]{GT2} and gives
 that the rank over $\Q(w)$ is equal to $2$.

  The triple used for the curve is as follows
   \begin{align*}
 q_1=&-\frac{(-2 + w^2) (1 - 6 w + w^2) (-5 + 6 w + w^2)}{6 (-1 + w) (3 + w) (-1 + 2 w) (-1 - 2 w + w^2)} ,\\
 q_2=&\frac{9 (-1 + w) (3 + w) (-1 + 2 w) (-1 - 2 w + w^2)}{ 2 (-2 + w^2) (1 - 6 w + w^2) (-5 + 6 w + w^2)},\\
    q_3=&-\frac{(-1 - 2 w + w^2) (-5 + 6 w + w^2) (-11 + 36 w - 10 w^2 - 12 w^3 + w^4)}{6 (-1 + w) (3 + w) (-1 + 2 w) (-2 + w^2) (1 - 6 w + w^2)}.
  \end{align*}

  For the value $w=\frac{7}{19}$ the corresponding curve has rank $6$ over $\Q$. This curve was previously found by Elkies and it is the record for this kind of curves (see \cite {D2} for details of this curve). Now we have shown that in can be obtained by Diophantine triples. In fact, it is induced by the Diophantine triple
 \[
  \left\{\frac{31269599}{31628160},-\frac{23721120}{31269599},\frac{
   1461969791}{7144352640}\right\}
   \]

   \subsection{Search for higher rank. Third curve}
 In a recent paper \cite{DKMS} it is shown the existence of infinitely many rational Diophantine  sextuples.  In the proof elliptic curves are used in order to extend a family of rational Diophantine triples to a family of sextuples. It is shown that the curve induced by that triple has rank $1$ and torsion group  $\Z /  2 \Z\times\Z / 6 \Z$  over $\Q(t)$. Again in this family we are able to find a specialization that gives a curve with rank $2$.

 The triple is as follows
   \begin{align*}
a=&\frac{ 18 t (t - 1) (t + 1)  }{  (t^2 - 6 t + 1) (t^2 + 6 t + 1)      } ,\\
b=&\frac{ (t - 1) (t^2 + 6 t + 1)^2 }{ 6 t (t + 1) (t^2 - 6 t + 1) } ,\\
c=&\frac{ (t + 1) (t^2 - 6 t + 1)^2 }{ 6 t (t - 1) (t^2 + 6 t + 1)} . \\
  \end{align*}
  As said before the elliptic curve induced by the triple $\{a, b, c\}$ has torsion group $\Z /  2 \Z\times\Z / 6 \Z$  and rank $\geq 1$ over $\Q(t)$.
  We have found two possibilities in order to increase the rank, one of them by solving $1-30 t+t^2=$ square;  this gives a rank $2$  curve, but it is one of the previously shown families.
  In the other condition we have to parametrize  $t^2-3t+1=$  square or $t=-\frac{(-1 + w) (1 + w)}{3 + 2 w}$.  When written in the form $Y^2=X^3+A_3(w) X^2+B_3(w) X$ the coefficients are
   \begin{align*}
 A_3(w)=&-1899008 - 5996544 w - 1469440 w^2 + 13980672 w^3 + 14383360 w^4\\& -
 6990336 w^5 - 16274176 w^6 - 3491328 w^7 + 5554240 w^8 +
 2814336 w^9 \\&- 219616 w^{10} - 160416 w^{11} + 120808 w^{12} + 38928 w^{13} -
 856 w^{14} + w^{16},\\
 B_3(w)=&\,\,3456 (-1 + w)^3 (1 + w)^3 (3 + 2 w)^3 (-2 + w^2)^3 (-14 - 12 w +
   w^2)^3 \\&(10 + 12 w + 2 w^2 + w^4) (-44 - 24 w + 56 w^2 + 36 w^3 +
   w^4).
  \end{align*}
  The triple becomes
 \begin{align*}
 q_1=&-\frac{18 (-1 + w) (1 + w) (3 + 2 w) (-4 - 2 w + w^2) (2 + 2 w + w^2)}{(-2 + w^2) (-14 - 12 w + w^2) (-8 + 20 w^2 + 12 w^3 + w^4) },\\
 q_2=&-\frac{(-2 + w^2)^2 (-14 - 12 w + w^2)^2 (2 + 2 w + w^2)}{6 (-1 + w) (1 + w) (3 + 2 w) (-4 - 2 w + w^2) (-8 + 20 w^2 + 12 w^3 +
   w^4)},\\
    q_3=&-\frac{(-4 - 2 w + w^2) (-8 + 20 w^2 + 12 w^3 + w^4)^2}{6 (-1 + w) (1 + w) (3 + 2 w) (-2 + w^2) (-14 - 12 w + w^2) (2 + 2 w +
   w^2)}
  \end{align*}
  and the $X$-coordinates of two points of infinite order are as follows:
   \begin{align*}
 x_1=&\frac{108 (-1 + w)^2 (1 + w)^2 (3 + 2 w)^2 (-2 + w^2)^3 (-14 - 12 w +
   w^2)^3 (2 + 2 w + w^2)^2}{(10 + 12 w + 2 w^2 + w^4)^2
   },\\
 x_2=&-27 (-1 + w) (1 + w) (3 + 2 w) (-2 + w^2)^2 (-14 - 12 w +
   w^2)^2 (-4 - 2 w + w^2)^2.
  \end{align*}
 The specialization $w=14$ satisfies the conditions of \cite[Theorem 1.1]{GT2},
 so the rank over $\Q(w)$ is equal to $2$.

  For $w=-39$ we get again Elkies' rank $6$ curve with equation
  \begin{align*}
 Y^2 + XY = X^3 - 37680956700999226080263982005713090640\,X \\
         \mbox{}- 36992898397926078743894505902555362159162611772488902400,
   \end{align*}
this time induced by the Diophantine triple
\[
\Big\{\frac{7567037280}{7833785281 },
\frac{4161669360289}{569762123040},
    \frac{1359453258559}{948852707040} \Big\}. \]
Let us mentioned that this curve was originally found in \cite{El2} by the sieving
within general family of curves with torsion $\Z /  2 \Z\times\Z / 6 \Z$.
With the same method, the authors and P. Tadi\'c recently found
a new curve with torsion $\Z /  2 \Z\times\Z / 6 \Z$ and rank 6:
   \begin{align*}
Y^2 + XY + Y = X^3 - 5012222351518888614250804048874855041913\,X \\
	    \mbox{}+ 136464417579052941096027626504118630642626009794008307407656.
   \end{align*}

\section{Torsion   $\Z /  2 \Z\times\Z / 8 \Z$}\label{sec5}
It has been shown in \cite{D4} (see also \cite{CG}) that every elliptic curve over $\Q$
with torsion $\Z /  2 \Z\times\Z / 8 \Z$ is induced by a Diophantine triple.
More precisely, any such curve is induced by a Diophantine triple of the form
\begin{equation}\label{z2z8T}
\left\{ \frac{2T}{T^2-1}, \,\, -\frac{1-T^2}{2T}, \,\, \frac{6T^2-T^4-1}{2T(T^2-1)} \right\}.
\end{equation}
Since elements of this triples clearly have mixed signs,
we may ask is there any such curve induced by a rational Diophantine triple
with all three positive elements.
We show that the answer in positive.

Consider the family of Diophantine triples
\begin{gather*}
 \Bigl\{\frac{2(2w^4+4w^3+2w^2+1)(1+2w)}{w(2w+2w^2-1)(2+w)(w-1)(w+1)},
-\frac{2(w+1)^2(w-1)^2}{(2+w)(2w+2w^2-1)w}, \\
 \frac{2w^2(2+w)^2}{(2w+2w^2-1)(w-1)(w+1)} \Bigr\}.
\end{gather*}
By comparing the $j$-invariant of the elliptic curve induced by this triple
with the $j$-invariant induced by (\ref{z2z8T}), we obtain the condition
$2w^4+4w^3+2w^2+1 =$ square. The quartic curve defined by this equation is birationally 
equivalent to the elliptic curve
$Y^2 = X^3-X^2-9X+9$
with rank 1 and a generator $P=[0, 3]$.
By taking certain multiples of $P$, e.g. $3P$, $5P$, $8P$,
we get Diophantine triples with all positive elements.
E.g. let us take the point $3P=[-360/169, -8211/2197]$. By transforming it back to the quartic, we get 
$w=26/89$, and by inserting this value for $w$ in the above family of Diophantine triples,  
we obtain the following Diophantine triple with all positive elements
$$\{a,b,c\}=\Bigl\{\frac{37471518967}{1381254420}, \frac{5832225}{571948}, \frac{6251648}{1562505} \Bigr\} $$
which induces the curve with torsion $\Z /  2 \Z\times\Z / 8 \Z$ and rank $1$.
In fact, in this way we get an infinite family of elliptic curves
with torsion $\Z /  2 \Z\times\Z / 8 \Z$ and positive rank
(since the point $[0,abc]$ is of infinite order).
This family is equivalent to the family $C_{13}$ from the paper by Lecacheux \cite{L2}.

It remains an open question whether or not there exists a Diophantine triple with all positive elements
inducing an elliptic curve with torsion $\Z /  2 \Z\times\Z / 8 \Z$ and rank $0$.
A candidate for such triple is given in \cite{DM,Milj}:
$$ \Bigl\{ \frac{1884586446094351}{25415891646864180},
\frac{14442883687791636}{7402559392524605},
\frac{60340495895762708555}{14487505263205637124} \Bigr\}, $$
for which \texttt{mwrank} and \texttt{magma} give that $0 \leq$ rank $\leq 2$
(and by the parity conjecture the rank should be 0 or 2).
The construction starts with the triple $\{a,b,c\}$, where
$b=\frac{r^2-1}{a}$, $c=a+b+2r$,
and we force the point with the first coordinate $\frac{1}{abc}$ to be of order 4.
The condition becomes $-4r^4+4r^2+1=$ square. The quartic curve defined by this equation is birationally
equivalent to the elliptic curve 
$Y^2 = X^3+X^2+X+1$ with rank 1 and the generator $P=[0, 1]$.
The first multiple of $P$ producing the triple with positive elements
is $6P$ (the next is $11P$), and it yields the triple given above.

\bigskip

{\bf Acknowledgement.} The authors would like to thank Ivica Gusi\'c, Filip Najman 
and the anonymous referee for very useful
comments on the previous version of this paper.

%%%%%%%%%%%%%%%%%%%%%%%%%%%
\end{document}